\documentclass[11pt,a4paper]{article}

\usepackage{latexsym,amssymb,amsfonts,amsmath,amsthm,epsfig,tabularx,stmaryrd}
\usepackage[usenames]{xcolor}
\usepackage{mathtools}

\newtheorem{theorem}{Theorem}
\theoremstyle{plain}
{\theoremstyle{definition}
\newtheorem{remark}{Remark}}

\newtheorem{GKSexample}{Example}

\newcommand{\bsy}{{\boldsymbol{y}}}
\newcommand{\bsz}{{\boldsymbol{z}}}
\newcommand{\bsA}{{\boldsymbol{A}}}

\newcommand{\bszero}{{\boldsymbol{0}}} 

\newcommand{\bbE}{{\mathbb{E}}}

\newcommand{\bbR}{{\mathbb{R}}}

\newcommand{\R}{{\mathbb{R}}} 

\newcommand{\calO}{{\mathcal{O}}}

\newcommand{\rd}{{\mathrm{d}}}

\graphicspath{{../../}{../}{./Figures/}}

\usepackage{pdfsync}
\definecolor{darkred}{RGB}{139,0,0}
\definecolor{darkgreen}{RGB}{0,100,0}
\definecolor{darkmagenta}{RGB}{139,0,139}
\definecolor{darkorange}{RGB}{220,110,20}

\newcommand{\GKSind}{\mathop{\rm ind}}

\title{Preintegration is not smoothing when monotonicity fails}
\author{Alexander D.~Gilbert, Frances Y.~Kuo and Ian H.~Sloan%
   \footnote{School of Mathematics and Statistics, UNSW Sydney, Sydney NSW 2052, Australia \newline
   Emails: alexander.gilbert@unsw.edu.au, f.kuo@unsw.edu.au, i.sloan@unsw.edu.au}
   }
\date{\today}

\begin{document}

\maketitle

\begin{abstract}
Preintegration is a technique for high-dimensional integration
over $d$-dimensional Euclidean space, which is designed to reduce an
integral whose integrand contains kinks or jumps to a $(d-1)$-dimensional
integral of a smooth function. The resulting smoothness allows efficient
evaluation of the $(d-1)$-dimensional integral by a Quasi-Monte Carlo or
Sparse Grid method. The technique is similar to conditional sampling in
statistical contexts, but the intention is different: in conditional
sampling the aim is to reduce the variance, rather than to achieve
smoothness.  Preintegration involves an initial integration with respect
to one well chosen real-valued variable. Griebel, Kuo, Sloan [\emph{Math.\
Comp.\ 82 (2013), 383--400}] and Griewank, Kuo, Le\"ovey, Sloan [\emph{J.\
Comput.\ Appl.\ Maths.\ 344 (2018), 259--274}] showed that the resulting
$(d-1)$-dimensional integrand is indeed smooth under appropriate
conditions, including a key assumption
--- the integrand of the smooth function underlying the kink or jump is
\emph{strictly monotone} with respect to the chosen special variable when
all other variables are held fixed. The question addressed in this paper
is whether this monotonicity property with respect to one well chosen
variable is necessary. We show here that the answer is essentially yes, in
the sense that without this property the resulting $(d-1)$-dimensional
integrand is generally not smooth, having square-root or other
singularities.
\end{abstract}

\section{Introduction} \label{sec:intro}

Preintegration is a method for numerical integration over $\bbR^d$, where
$d$ may be large, in the presence of ``kinks'' (i.e., discontinuities in
the gradients) or ``jumps'' (i.e., discontinuities in the function
values). In this method one of the variables is integrated out in a
``preintegration'' step, with the aim of creating a smooth integrand over
$\bbR^{d-1}$, smoothness being important if the intention is to
approximate the $(d-1)$-dimensional integral by a method that relies on
some smoothness of the integrand, such as the Quasi-Monte Carlo (QMC)
method \cite{DKS13} or Sparse Grid (SG) method \cite{BG04}.

Integrands with kinks and jumps arise in option pricing, because an option
is normally considered worthless if the value falls below a predetermined
strike price. In the case of a continuous payoff function this introduces
a kink, while in the case of a binary or other digital option it
introduces a jump. Integrands with jumps also arise in computations of
cumulative probability distributions, see \cite{GKS21}.

In this paper we consider the version of preintegration for functions with
kinks or jumps presented in the recent papers \cite{GKS13,GKSnote,GKLS18},
in which the emphasis was on a rigorous proof of smoothness of the
preintegrated $(d-1)$-dimensional integrand, in the sense of proving
membership of a certain mixed derivative Sobolev space, under appropriate
conditions.

A key assumption in \cite{GKS13,GKSnote,GKLS18} was that the smooth
function (the function $\phi$ in \eqref{problem2} below) underlying the
kink or jump is \emph{strictly monotone} with respect to the special
variable chosen for the preintegration step, when all other variables are
held fixed. While a satisfactory analysis was obtained under that
assumption, it was not clear from the analysis in
\cite{GKS13,GKSnote,GKLS18} whether or not the monotonicity assumption is
in some sense necessary. That is the question we address in the present
paper. The short answer is that \emph{the monotonicity condition is
necessary, in that in the absence of monotonicity the integrand typically
has square-root or other singularities}.

\subsection{Related work}

A similar method has already appeared as a practical tool in many other
papers, often under the heading ``conditional sampling'', see
\cite{GlaSta01}, Lemma~7.2 and preceding comments in \cite{ACN13a}, and a
recent paper \cite{LPB21} by L'Ecuyer and colleagues. Also relevant are
root-finding strategies for identifying where the payoff is positive, see
a remark in \cite{ACN13b} and \cite{Hol11,NuyWat12}. For other
``smoothing'' methods, see \cite{BST17,WWH17}.

The goal in conditional sampling is to decrease the variance of the
integrand, motivated by the idea that if the Monte Carlo method is the
chosen method for evaluating the integral then reducing the variance will
certainly reduce the root mean square expected error. The reality of
variance reduction in the preintegration context was explored analytically
in Section 4 of \cite{GKLS18}. But if cubature methods are used that
depend on smoothness of the integrand, as with QMC and SG methods, then
variance reduction is not the only consideration. In the present work the
focus is on smoothness of the resulting integrand.

\subsection{The problem}

For the rest of the paper we will follow the setting of \cite{GKLS18}. The
problem addressed in \cite{GKLS18} was the approximate evaluation of the
$d$-dimensional integral
\begin{equation} \label{problem1}
 I_d f
 \,:=\, \int_{\bbR^d}f(\bsy)\rho_d(\bsy)\,\rd\bsy
 \,=\, \int_{-\infty}^\infty\ldots\int_{-\infty}^\infty
       f(y_1,\ldots,y_d)\,\rho_d(\bsy)\,\rd y_1 \cdots\rd y_d,
\end{equation}
with
\[
  \rho_d(\bsy) \,:=\, \prod_{k=1}^d \rho(y_k),
\]
where $\rho$ is a continuous and strictly positive probability density
function on $\bbR$ with some smoothness, and $f$ is a real-valued function
of the form
\begin{equation} \label{problem2}
 f(\bsy) \,=\, \theta(\bsy)\,\GKSind\big(\phi(\bsy)\big),
\end{equation}
or more generally
\begin{equation} \label{problem2t}
 f(\bsy) \,=\, \theta(\bsy)\,\GKSind\big(\phi(\bsy)-t\big),
\end{equation}
where $\theta$ and $\phi$ are somewhat smooth functions, $\GKSind(\cdot)$
is the indicator function which has the value~$1$ if the argument is
positive and~$0$ otherwise, and $t$ is an arbitrary real number. When $t =
0$ and $\theta = \phi$ we have $f(\bsy) = \max(\phi(\bsy),0)$ and thus we
have the familiar kink seen in option pricing through the occurrence of a
strike price. When $\theta$ and $\phi$ are different (for example, when
$\theta(\bsy) = 1$) we have a structure that includes digital options.

The key assumption on the smooth function $\phi$ in \cite{GKLS18} was that
it has a positive partial derivative with respect to some well chosen
variable $y_j$ (and so is an increasing function of $y_j$); that is, we
assume that for the special choice of $j\in \{1,\ldots,d\}$  we have
\begin{equation} \label{monotone}
 \frac{\partial \phi}{\partial y_j}(\bsy) > 0
 \qquad \mbox{for all} \quad
 \bsy \in \bbR^d.
\end{equation}
In other words, $\phi$ is monotone increasing with respect to $y_j$ when
all variables other than $y_j$ are held fixed.

With the variable $y_j$ chosen to satisfy this condition, the
preintegration step is to evaluate
\begin{equation}\label{firstPj}
  (P_j f)(\bsy_{-j})
  \,:=\, \int_{-\infty}^\infty f(y_j,\bsy_{-j})\,\rho(y_j)\,\rd y_j,
\end{equation}
where $\bsy_{-j}\in\bbR^{d-1}$ denotes all the components of $\bsy$ other
than $y_j$. Once $(P_j f)(\bsy_{-j})$ is known we can evaluate $I_d f$ as
the $(d-1)$-dimensional integral
\begin{equation}\label{eq:reducedint}
 I_d f \,=\, \int_{\bbR^{d-1}}(P_j f)(\bsy_{-j})\,\rho_{d-1}(\bsy_{-j})\,\rd\bsy_{-j},
\end{equation}
which can be done efficiently if $(P_j f)(\bsy_{-j})$ is smooth. In the
implementation of preintegration, note that if the integral
\eqref{eq:reducedint} is to be evaluated by an $N$-point cubature rule,
then the preintegration step in \eqref{firstPj} needs to be carried out
for $N$ different values of $\bsy_{-j}$.

The key is the preintegration step. Because of the monotonicity assumption
\eqref{monotone}, for each $\bsy_{-j}\in \bbR^{d-1}$ there is at most one
value of the integration variable $y_j$ such that $\phi(y_j, \bsy_{-j}) =
t$. We denote that value of $y_j$, if it exists, by $\xi(\bsy_{-j})
=\xi_t(\bsy_{-j})$ so that $\phi(\xi(\bsy_{-j}), \bsy_{-j}) =
t$. Under the condition \eqref{monotone} it follows from the implicit
function theorem that $\xi(\bsy_{-j})$ is smooth if $\phi$ is smooth.
Then we can write the preintegration step as
\begin{equation}\label{P1psi}
  (P_j f)(\bsy_{-j})
  \,=\,
  \int_{\xi(\bsy_{-j})}^\infty
  \theta\big(y_j,\bsy_{-j}\big)\,\rho(y_j)\,\rd y_j,
\end{equation}
 which is a smooth function of $\bsy_{-j}$ if $\theta$ is smooth.

\subsection{Informative examples} \label{sec:eg}

We now illustrate the success and failure of the preintegration process
with some simple examples. In these examples we take $d=2$ and $t=0$, and
choose $\rho$ to be the standard normal probability density, $\rho(y) =
\exp(-y^2/2)/\sqrt{2\pi}$. We also initially take $\theta(y_1,y_2) = 1$,
and comment on other choices at the end of the section.

\begin{figure}[t]
\begin{center}
\includegraphics[width=\textwidth]{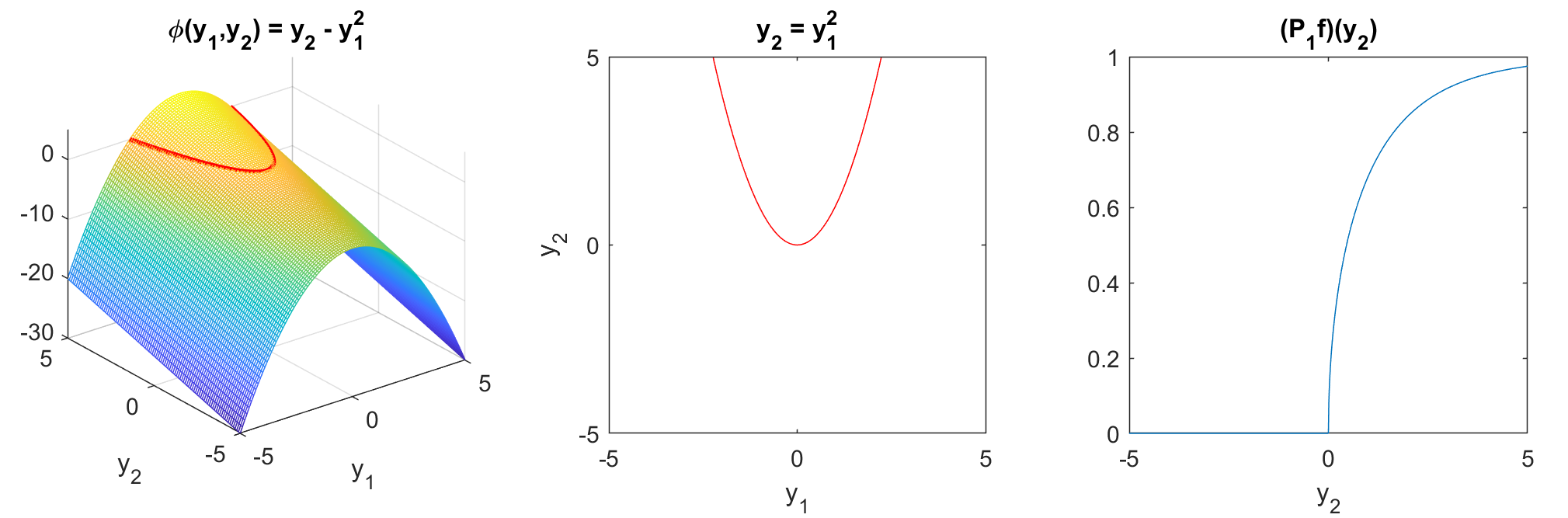}
\end{center}
\caption{Illustrations for Example~\ref{eg1}.} \label{fig:eg1}
\end{figure}

\begin{GKSexample} \label{eg1}
In this example we take
\[
  \phi(y_1, y_2 ) \,=\, y_2 -  y_1^2,
\]
see Figure~\ref{fig:eg1} (left). The zero set of this function is the
parabolic curve $y_2 = y_1^2$, see Figure~\ref{fig:eg1} (middle). The
positivity set of $\phi$ (i.e., the set for which $f(\bsy)$ defined by
\eqref{problem2} is non-zero) is the open region above the parabola.

If we take the special variable to be $y_2$ (i.e., if we take $j=2$) then
the monotonicity condition~\eqref{monotone} is satisfied, and the
preintegration step is truly smoothing: specifically, we see that
\[
 (P_2 f)(y_1) \,=\, \int_{y_1^2}^\infty \rho(y_2)\,\rd y_2
 \,=\, 1-\Phi(y_1^2),
\]
where $\Phi(x) := \int_{-\infty}^x \rho(y) \,\rd y$ is the standard normal
cumulative distribution. Thus $(P_2 f)(y_1)$ is a smooth function for
all $y_1 \in \bbR$, and $I_2f$ is the integral of a smooth integrand over
the real line,
\[
 I_2 f \,=\, \int_{-\infty}^\infty (P_2 f)(y_1)\, \rho(y_1) \,\rd y_1
 \,=\, \int_{-\infty}^\infty \big(1-\Phi(y_1^2)\big)\, \rho(y_1) \,\rd y_1.
\]

If on the other hand we take the special variable to be $y_1$ (i.e., take
$j=1$) then we have
\[
 (P_1 f)(y_2) \,=\,
 \begin{cases}
 0 & \mbox{if  } y_2 \le 0,\\
 \displaystyle\int_{-\sqrt{y_2}}^{\sqrt{y_2}} \rho(y_1)\,\rd y_1
 \,=\, \Phi(\sqrt{y_2}) - \Phi(-\sqrt{y_2})
 & \mbox{if  } y_2 > 0.
 \end{cases}
\]
The graph of $(P_1 f)(y_2)$, shown in Figure~\ref{fig:eg1} (right),
reveals that there is a singularity at $y_2= 0$. To see the nature of the
singularity, note that since $\rho(y_1) = \rho(0)\exp(-y_1^2/2) = \rho(0)
+ \calO(y_1^2)$ as $y_1 \to 0$, we can write
\begin{equation}\label{example1P1}
 (P_1 f)(y_2) \,=\,
 \begin{cases}
 0 & \mbox{if  } y_2 \le 0,\\
 \displaystyle\int_{-\sqrt{y_2}}^{\sqrt{y_2}} \rho(y_1)\,\rd y_1
 \,=\, 2\sqrt{y_2}\, \rho(0) + \calO\big(y_2^{3/2}\big)
 & \mbox{if  } y_2 > 0.
 \end{cases}
\end{equation}
Thus in this simple example $(P_1 f)(y_2) $ is not at all a smooth
function of $y_2$, having a square-root singularity, and hence an infinite
one-sided derivative, at $y_2 = 0$.
\end{GKSexample}

\begin{figure}[t]
\begin{center}
\includegraphics[width=\textwidth]{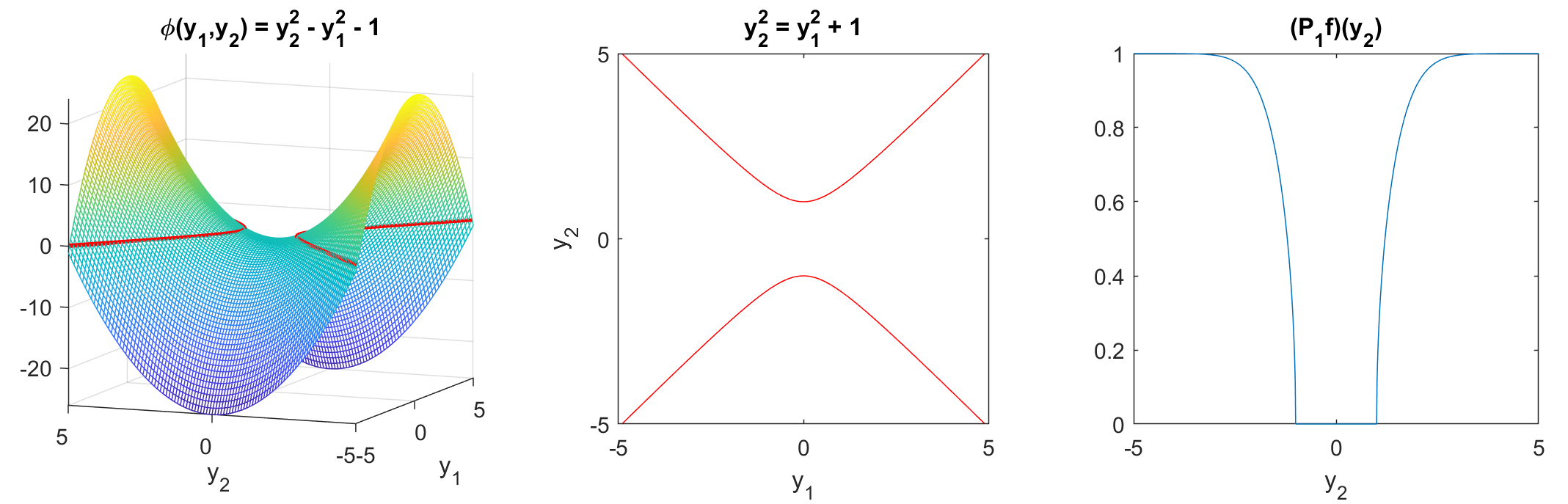}
\end{center}
\caption{Illustrations for Example~\ref{eg2}.} \label{fig:eg2}
\end{figure}

\begin{GKSexample} \label{eg2}
In this example we take
\[
  \phi(y_1,y_2) \,=\, y_2^2 - y_1^2 -1,
\]
see Figure~\ref{fig:eg2} (left). The zero set of $\phi$ is now the
hyperbola $y_2^2 = y_1^2 + 1$, see Figure~\ref{fig:eg2} (middle), and
the positivity set is the union of the open regions above and below the
upper and lower branches respectively. Taking $j = 1$, we see that
monotonicity again fails, and that specifically
\begin{align*} 
 &(P_1f)(y_2) \nonumber\\
 &\,=\,
 \begin{cases}
 0 &\mbox{if } y_2 \in [-1,1], \\
 \displaystyle\int_{-\sqrt{y_2^2 - 1}}^{\sqrt{y_2^2 - 1}} \rho(y_1) \,\rd y_1
 \,=\, 2 \sqrt{y_2^2 - 1} \, \rho(0) + \calO\big((y_2^2-1)^{3/2}\big)
 & \mbox{if } |y_2|>1,
 \end{cases}
\end{align*}
the graph of which is shown in Figure~\ref{fig:eg2} (right). Again we see
square-root singularities, this time two of them.
\end{GKSexample}

\begin{GKSexample} \label{eg3}
Here we take
\[
  \phi(y_1,y_2) \,=\, y_2^2 - y_1^2,
\]
see Figure~\ref{fig:eg3} (left). The level set is now the pair of lines
$y_2 = \pm y_1$, see Figure~\ref{fig:eg3} (middle), and the
positivity set is the open region above and below the crossed lines. This
time $P_1 f$ is given by
\begin{align*}
 (P_1 f) (y_2)
 \,=\, \int_{-|y_2|}^{|y_2|} \rho(y_1)\,\rd y_1
 \,=\, 2|y_2| \, \rho(0) + \calO\big(|y_2|^3\big),
\end{align*}
revealing in Figure~\ref{fig:eg3} (right) a different kind of singularity
(a simple discontinuity in the first derivative), but one still
unfavorable for numerical integration.
\end{GKSexample}

Example~\ref{eg3} is rather special, in that the preintegration is
performed on a line that touches a saddle at its critical point (the
``flat'' point'' of the saddle). Example~\ref{eg4} below illustrates
another situation, one that is in some ways similar to Example 1, but one
perhaps less likely to be seen in practice.

\begin{figure}[t]
\begin{center}
\includegraphics[width=\textwidth]{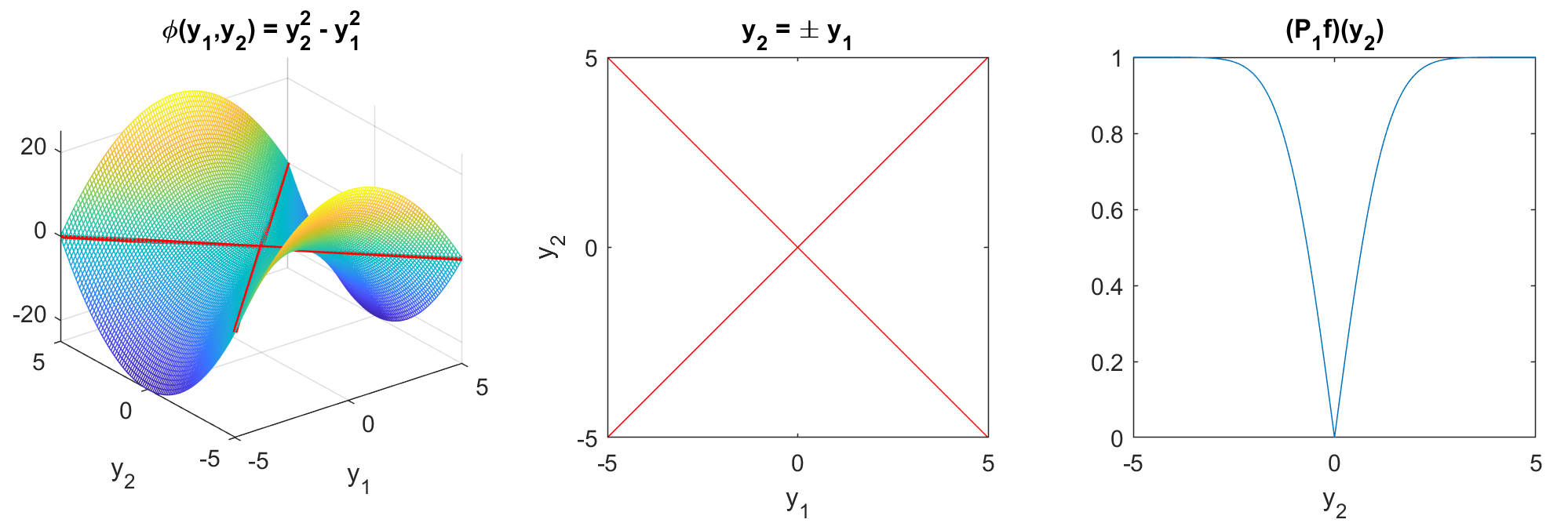}
\end{center}
\caption{Illustrations for Example~\ref{eg3}.} \label{fig:eg3}
\end{figure}

\begin{figure}[t]
\begin{center}
\includegraphics[width=\textwidth]{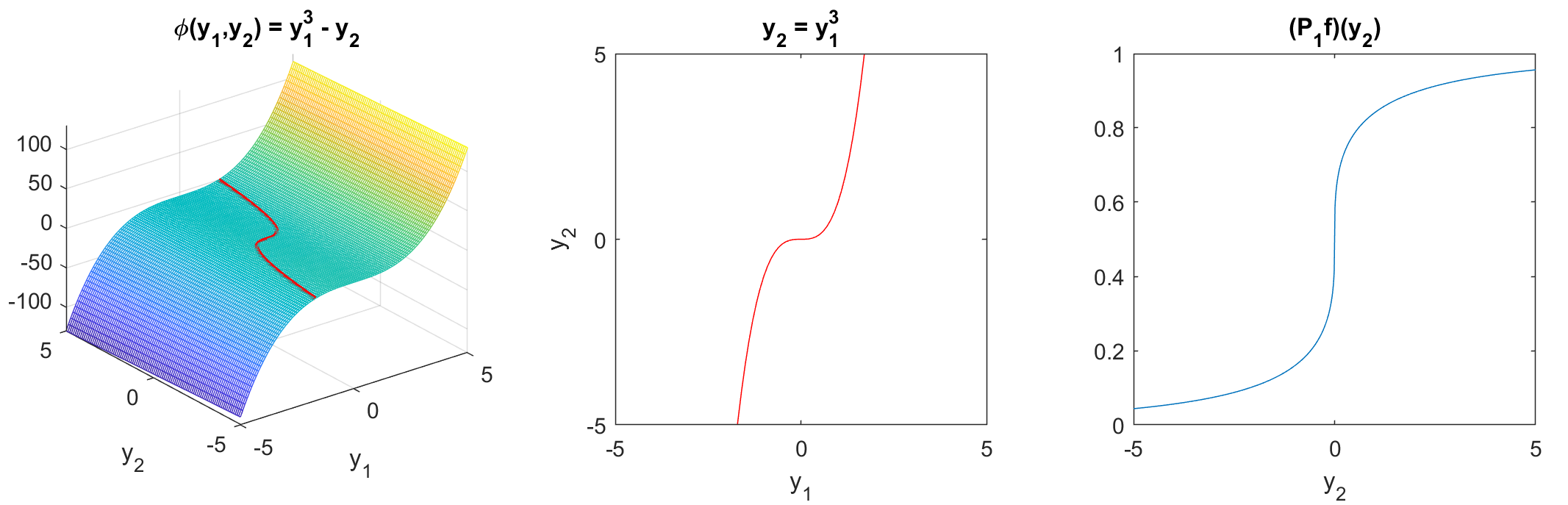}
\end{center}
\caption{Illustrations for Example~\ref{eg4}.} \label{fig:eg4}
\end{figure}

\begin{GKSexample} \label{eg4}
Here we consider
\[
 \phi(y_1, y_2) \,=\, y_1^3 -y_2,
\]
see Figure~\ref{fig:eg4} (left). The zero level set of $\phi$ is the graph
of $y_2 = y_1^3$, see Figure~\ref{fig:eg4} (middle), and the positivity
set is the unbounded domain to the right of the curve. We see that
\begin{align*}
 (P_1 f)(y_2)
 \,=\, \int_{-\infty}^{y_2^{1/3}}\rho(y_1) \,\rd y_1
 &\,=\, \int_{-\infty} ^0\rho(y_1) \,\rd y_1 + \int_0^{y_2^{1/3}}\rho(y_1) \,\rd y_1 \\
 &\,=\, \frac{1}{2} + y_2^{1/3}\rho(0) + \calO\big(|y_2|\big),
\end{align*}
which holds regardless of the sign of $y_2$. The graph of $P_1f$ in
Figure~\ref{fig:eg4} (right) displays the cube-root singularity at $y_2 =
0$.
\end{GKSexample}

In each of the above examples we took $\theta(y_1,y_2)= 1$. Other choices
for $\theta$ are generally not interesting, as they do not affect the
nature of the singularity. An exception is the choice $\theta(y_1,y_2) =
\phi(y_1,y_2)$, which yields a kink rather than a jump because
\[
\phi(\bsy)\, \GKSind(\phi(\bsy)) \,=\, \max(\phi(\bsy),0),
\]
and so leads to a weaker singularity. For example, for $f(y_1,y_2) =
\max(\phi(y_1,y_2),0)$ with $\phi$ as in Example~\ref{eg1}, we obtain
instead of \eqref{example1P1}
\[
 (P_1 f)(y_2)
 \,=\,
 \begin{cases}
 0 & \mbox{if } y_2 \le 0,\\
 \displaystyle\int_{-\sqrt{y_2}}^{\sqrt{y_2}} (y_2-y_1^2)\,\rho(y_1)\,\rd y_1 \,=\,
 \frac{4}{3}\, y_2^{3/2}\,\rho(0) + \calO\big(y_2^{5/2}\big) & \mbox{if } y_2 > 0.
 \end{cases}
\]
With the recognition that kinks lead to less severe singularities
than jumps, but located at the same places, from now on we shall for
simplicity consider only the case $\theta(\bsy) = 1$.

\subsection{Outline of this paper}

In Section~\ref{sec:smooth} we study theoretically the smoothness of the
preintegrated function, assuming that the original $d$-variate function is
$f(\bsy) = \mathrm{ind}(\phi(\bsy) - t)$, with $\phi$ smooth but not
monotone. We prove that the behavior seen in the above informative
examples is typical. Section~\ref{sec:num} contains a
numerical experiment for a  high-dimensional integrand that allows both monotone and
non-monotone choices for the preintegrated variable.
Section~\ref{sec:conc} gives brief conclusions.

\section{Smoothness theorems in $d$ dimensions} \label{sec:smooth}

In the general $d$-dimensional setting we take $\theta \equiv 1$, and
use the general form \eqref{problem2t} with arbitrary $t\in\bbR$.
Thus now we consider
\begin{equation}\label{general}
 f(\bsy) \,:=\, f_t(\bsy) \,:=\, \GKSind\big(\phi(\bsy) - t\big), \quad \bsy \in \bbR^d.
\end{equation}
A natural setting in which $t$ can take any value is in the computation of
the (complementary) cumulative probability distribution of
a random variable $X = \phi(\bsy)$, as in
\cite{GKS21}. In the case of option pricing varying $t$ corresponds to
varying the strike price.

For simplicity in this section we shall take the special
preintegration variable to be $y_1$, so fixing $j = 1$.  The question is
then, assuming that $\phi$ in \eqref{general} has smoothness at least
$C^2(\bbR^d)$, whether or not $P_1f_t$ given by
\[
 (P_1 f_t)(\bsy_{-1})
 \,\coloneqq\, \int_{-\infty}^\infty f_t(y_1,\bsy_{-1})\, \rho(y_1) \,\rd y_1
 \,=\, \int_{-\infty}^\infty \GKSind\big(\phi(\bsy) - t\big)\, \rho(y_1) \,\rd y_1
\]
is a smooth function of $\bsy_{-1}\in\bbR^{d-1}$.

To gain a first insight into the role of the parameter $t$ in
\eqref{general}, it is useful to observe that for the examples in
Section~\ref{sec:eg} a variation in $t$ can change the position and even
the nature of the singularity in $P_1 f_t$, but does not necessarily
eliminate the singularity. For a general $t\in\bbR$ and $\phi$ as in
Example~\ref{eg1}, we easily find that \eqref{example1P1} is replaced by
\begin{equation*} 
 (P_1 f_t)(y_2)
 \,=\,
 \begin{cases}
 0 &\mbox{if } y_2 \le t,\\
 \displaystyle\int_{-\sqrt{y_2-t}}^{\sqrt{y_2-t}} \rho(y_1) \,\rd y_1 &\mbox{if } y_2 > t,
 \end{cases}
\end{equation*}
so that the graph of $P_1 f_t$ is simply translated with the
singularity now occurring at $y_2 = t$ instead of $y_2=0$. The
situation is the same for $\phi$ as in Example~\ref{eg4}.

For $\phi$ as in Example~\ref{eg2}, the choice $t = -1$ recovers
Example~\ref{eg3}, while for $t > -1$ we find
\[
 (P_1f_t)(y_2)
 \,=\,
 \begin{cases}
 0 &\mbox{if } y_2 \in [-\sqrt{1+t},\sqrt{1+t}\,],\\
 \displaystyle\int_{-\sqrt{y_2^2 - 1 - t}}^{\sqrt{y_2^2 - 1 - t}} \rho(y_1) \,\rd y_1
 &\mbox{if } |y_2|>\sqrt{1+t}\,,\\
 \end{cases}
\]
thus in this case $(P_1 f_t)(y_2)$ has square-root singularities at $y_2 =
\pm\sqrt{1+t}$. For the case $t < -1$ (which we leave to the reader) $P_1
f_t$ has no singularity.

In \cite{GKLS18} it was proved that $P_1 f_t$ has the same smoothness
as $\phi$, provided that
\begin{equation}\label{monotone1}
\frac{\partial \phi}{\partial y_1}(\bsy) > 0\qquad\mbox{for all}\quad \bsy \in \bbR^d,
\end{equation}
together with some other technical conditions, see \cite[Theorems~2
and~3]{GKLS18}.

Here we are interested in the situation when $\phi$ is \textbf{not}
monotone increasing with respect to $y_1$ for all $\bsy_{-1}$.  In that
case (unless $\phi$ is always monotone decreasing) there is at least one
point, say $\bsy^* = (y_1^*,\bsy_{-1}^*)\in \bbR^d$, at which
$(\partial \phi/\partial y_1)(\bsy^*) = 0$. At such a point the
gradient of $\phi$ is either zero or orthogonal to the $y_1$ axis. If $t$
in \eqref{general} has the value $t=\phi(\bsy^*)$ then there is
generically a singularity of some kind in $P_1 f_t$ at the point
$\bsy^*_{-1}\in \bbR^{d-1}$. If $t\ne \phi(\bsy^*)$ then there is in
general no singularity in $P_1 f_t$ at the point $\bsy^*_{-1}\in
\bbR^{d-1}$, but if $t$ is not constant then the risk of encountering a
near-singularity is high.

The following theorem states a general result for the existence and the
nature of the singularities induced in $P_1 f_t$ in the common situation
in which the second derivative of $\phi$ with respect to $y_1$ is non-zero
at $\bsy = \bsy^*$, the point at which the first derivative with respect
to $y_1$ is zero.

\begin{theorem} \label{thm:sqrt}
Let $\phi \in C^2(\bbR^d)$, and assume that $\bsy^* =
(y_1^*,\bsy_{-1}^*)\in \bbR^d$ is such that
\begin{align} \label{eq:cond1}
 \frac{\partial \phi}{\partial y_1}(\bsy^*) = 0, \quad
 \frac{\partial^2 \phi}{\partial y_1^2}(\bsy^*) \ne 0,
 \quad\mbox{and}\quad
 \nabla \phi(\bsy^*)\ne \bszero.
\end{align}
Define $t := \phi(\bsy^*)$. Then $(P_1 f_t)(\bsy_{-1})$ has a square-root
singularity at $\bsy^*_{-1} \in \bbR^{d-1}$, similar to those in
Examples~\ref{eg1} and~\ref{eg2}, along any line in $\bbR^{d-1}$ through
$\bsy^*_{-1}$ that is not orthogonal to $\nabla_{\!-1}\phi(\bsy^*) :=
((\partial\phi/\partial y_2)(\bsy^*),(\partial\phi/\partial
y_3)(\bsy^*),\ldots,(\partial\phi/\partial y_d)(\bsy^*))$.
\end{theorem}

\begin{proof}
Since $\nabla \phi(\bsy^*)$ is not zero, and has no component in the
direction of the $y_1$ axis, it follows that $\nabla \phi(\bsy^*)$ can be
written as $(0,\nabla_{\!-1}\phi(\bsy^*))$, where
$\nabla_{\!-1}\phi(\bsy^*)= \nabla_{\!-1}\phi(y_1^*,\bsy_{-1}^*)$ is a
non-zero vector in $\bbR^{d-1}$
orthogonal to the $y_1$ axis. Note that as $\bsy_{-1}$ changes in a
neighborhood of $\bsy_{-1}^*$, the function $\phi(y_1^*,\bsy_{-1})$ is
increasing in the direction $\nabla_{\!-1}\phi(\bsy^*)$, and also in the
direction of an arbitrary unit vector $\bsz$ in $\bbR^{d-1}$ that has a
positive inner product with $\nabla_{\!-1}\phi(\bsy^*)$.  Our aim now is
to explore the nature of $P_1 f_t$ on the line through
$\bsy_{-1}^*\in\bbR^{d-1}$ that is parallel to any such unit vector
$\bsz$.

For simplicity of presentation, and without loss of generality, we assume
from now on that the unit vector $\bsz$ points in the direction of the
positive $y_2$ axis. This allows us to write $\bsy = (y_1, y_2, y^*_3,
\ldots,y^*_d) =: (y_1, y_2)$, temporarily ignoring all components other
than the first two. In this $2$-dimensional setting we know that
\[
 \frac{\partial \phi}{\partial y_1}(y_1^*,y_2^*) = 0, \quad
 \frac{\partial \phi}{\partial y_2}(y_1^*,y_2^*) > 0,
 \quad\mbox{and}\quad
 \frac{\partial^2 \phi}{\partial y_1^2}(y_1^*,y_2^*) \ne 0.
\]
Since $(\partial \phi/\partial y_2)(y_1,y_2)$ is continuous and positive
at $(y_1,y_2) = (y_1^*,y_2^*)$, it follows that $(\partial \phi/\partial
y_2)(y_1,y_2)$ is positive on the rectangle $[y^*_1 -\delta, y^*_1 +
\delta] \times [y^*_2 -\delta, y^*_2 + \delta]$ for sufficiently small
$\delta > 0$. Since $\phi(y_1,y_2)$ is increasing with respect to $y_2$ on
this rectangle, it follows that for each $y_1 \in [y^*_1 - \delta,
y^*_1 + \delta]$ there is at most one value of $y_2$ such that
$\phi(y_1,y_2) = t$, and further, that for $|y_1 - y^*_1|$ sufficiently
small there is exactly one value of $y_2$ such that $\phi(y_1,y_2) = t$.
For that unique value we write $y_2 = \zeta(y_1)$, hence by construction
we have $\phi(y_1, \zeta(y_1))= t$, and $\zeta(y^*_1) = y^*_2$.

From the implicit function theorem (or by implicit differentiation of
$\phi(y_1, \zeta(y_1)) = t$ with respect to $y_1$) we obtain
\begin{align} \label{eq:diff1}
 \zeta'(y_1)
 \,=\, -\, \frac{(\partial\phi/\partial y_1)(y_1,\zeta(y_1))}
         {(\partial\phi/\partial y_2)(y_1,\zeta(y_1))},
\end{align}
in which the denominator is non-zero in a neighborhood of $y^*_1$.  From
this it follows that
\[
 \zeta'(y^*_1) \,=\, 0.
\]
Differentiating \eqref{eq:diff1} using the product rule and the chain rule
and then setting $y_1 = y^*_1$ (so that several terms vanish), we obtain
\[
 \zeta''(y^*_1) \,=\, -\, \frac{(\partial^2\phi/\partial y_1^2)(y^*_1, y^*_2)}
 {(\partial\phi/\partial y_2)(y^*_1, y^*_2)},
\]
which by assumption is non-zero. Below we assume that $(\partial^2
\phi/\partial y_1^2)(y^*_1, y^*_2) < 0$, from which it follows that
$\zeta''(y^*_1)$ is positive; the case 
$(\partial^2 \phi/\partial y_1^2)(y^*_1, y^*_2) > 0$ is similar. Taylor's
theorem with remainder gives
\begin{align}\label{Taylor}
 \zeta(y_1)
 &\,=\, \zeta(y^*_1) + \tfrac{1}{2}(y_1 - y^*_1)^2\, \zeta''(y^*_1)\,(1 + o(1)) \nonumber\\
 &\,=\, y^*_2 + \tfrac{1}{2}(y_1 - y^*_1)^2\, \zeta''(y^*_1)\,(1 + o(1)),
\end{align}
where $o(1) \to 0$ as $|y_1 - y^*_1| \to 0$, thus $\zeta(y_1)$ is a convex
function in a neighborhood of $y^*_1$.

Given $y_2$ in a neighborhood of $y^*_2$, our task now is to evaluate the
contribution to the integral
\[
 (P_1 f_t)(y_2) \,=\, \int_{-\infty}^\infty \GKSind\big(\phi(y_1,y_2) - t\big)
 \rho(y_1) \,\rd y_1
\]
from a neighborhood of $y^*_1$. Thus we need to find the set of $y_1$
values in a neighborhood of $y^*_1$ for which $\phi(y_1,y_2) > t$. Because
of \eqref{Taylor}, the set is either empty, or is the open interval with
extreme points given by the solutions of $\zeta(y_1) = y_2$, i.e.,
\[
 y^*_2 + \tfrac{1}{2}(y_1 - y^*_1)^2\, \zeta''(y^*_1)\,(1 + o(1)) \,=\, y_2,
\]
implying
\[
 (y_1 - y^*_1)^2
 \,=\, \frac{2(y_2 - y^*_2)}{\zeta''(y^*_1)\,(1 + o(1))}
 \,=\, \frac{2(y_2 - y^*_2)}{\zeta''(y^*_1)} (1 + o(1)).
\]
There is no solution for $y_2 < y^*_2$, while for $y_2 > y^*_2$ the
solutions are
\[
 y_1 \,=\, y^*_1 \pm c \sqrt{y_2 - y^*_2}\, (1 + o(1)),
\]
with $c:= \sqrt{2/\zeta''(y^*_1)}$. Thus the contribution to $P_1
f_t(y_2)$ from the neighborhood of $y^*_2$ is zero for $y_2 < y^*_2$, and
for $y_2 > y^*_2$ is
\[
 \int_{y^*_1 - c\sqrt{y_2 - y^*_2}\,(1+o(1))}^{y^*_1 + c\sqrt{y_2 - y^*_2}\,(1+o(1))} \rho(y_1)\,\rd y_1
 \,=\, 2 c\sqrt{y_2 - y^*_2} \,\rho(y^*_1) \,(1 + o(1)).
\]
Thus there is a singularity in $P_1 f_t$ of exactly the same character as
that in Examples~\ref{eg1} and~\ref{eg2}.
\end{proof}

\begin{remark}
\label{rem:sqrt}
We now return to consider the examples in Section~\ref{sec:eg} in the
context of Theorem~\ref{thm:sqrt}.

\begin{itemize}
\item For $\phi$ as in Example~\ref{eg1}, we have $(\partial
    \phi/\partial y_1)(y_1,y_2) = -2y_1$, $(\partial^2 \phi/\partial
    y_1^2)(y_1,y_2) = -2\ne 0$, and $\nabla\phi(y_1,y_2) =
    (-2y_1,1)\ne (0,0)$. Thus \eqref{eq:cond1} holds, e.g., with
    $\bsy^* = (0,0)$ and $t = \phi(\bsy^*) = 0$, and $P_1 f_t$ indeed
    displays the predicted square-root singularity at $y_2 = 0$, see Figure~\ref{fig:eg1}.

\item For $\phi$ as in Example~\ref{eg2}, we have $(\partial
    \phi/\partial y_1)(y_1,y_2) = -2y_1$, $(\partial^2 \phi/\partial
    y_1^2)(y_1,y_2) = -2\ne 0$, and $\nabla\phi(y_1,y_2) =
    (-2y_1,2y_2)$. Thus \eqref{eq:cond1} holds, e.g., with $\bsy^* =
    (0,\pm 1)$ and $t = \phi(\bsy^*) = 0$, and $P_1 f_t$ indeed shows
    the predicted square-root singularities at $y_2 = \pm 1$, see Figure~\ref{fig:eg2}.

\item For $\phi$ as in Example~\ref{eg3}, we have the same derivative
    expressions as in Example~\ref{eg2}. Thus \eqref{eq:cond1} holds,
    e.g., again with $\bsy^* = (0,\pm 1)$, but now $t = \phi(\bsy^*) =
    1$, and we effectively recover Example~\ref{eg2} with square-root
    singularities for $P_1 f_t$ at $y_2 = \pm 1$.
    However, if we consider instead the point $\bsy^\dagger = (0,0)$
    and $t = \phi(\bsy^\dagger)=0$, as in Figure~\ref{fig:eg3}, then
    we have $\nabla\phi(\bsy^\dagger) = 0$ so the non-vanishing
    gradient condition in \eqref{eq:cond1} fails and
    Theorem~\ref{thm:sqrt} does not apply at this point
    $\bsy^\dagger$. In this case we actually observe an absolute-value
    singularity for $P_1 f_t$ at $y_2=0$ rather than a square-root
    singularity.

\item For $\phi$ as in Example~\ref{eg4}, we have $(\partial
    \phi/\partial y_1)(y_1,y_2) = 3y_1^2$, $(\partial^2 \phi/\partial
    y_1^2)(y_1,y_2) = 6y_1$, and $\nabla\phi(y_1,y_2) = (3y_1^2,-1)\ne
    (0,0)$. It is impossible to satisfy both the first and second
    conditions in \eqref{eq:cond1} so Theorem~\ref{thm:sqrt} does not
    apply anywhere. In particular, at the point $\bsy^\dagger = (0,0)$
    and $t=\phi(\bsy^\dagger)=0$, as in Figure~\ref{fig:eg4},
    we have $(\partial^2 \phi/\partial y_1^2)(\bsy^\dagger) = 0$,
    and in consequence (given that the
    third derivative does not vanish) $P_1 f_t$ has a cube-root
    singularity at $y_2 = 0$ rather than a square-root singularity.
\end{itemize}
\end{remark}

From Theorem~\ref{thm:sqrt} one might suspect, because $t$ in the theorem
has the particular value $t = \phi(\bsy^*)$, that singularities of this
kind are rare. However, in the following theorem we show that values of $t
\in \bbR$ at which singularities occur in $P_1 f_t$ are generally not
isolated. This is essentially because points at which
$(\partial\phi/\partial y_1)(\bsy) = 0$ are themselves generally not
isolated.

\begin{theorem} \label{thm:not-iso}
Let $\phi \in C^2(\bbR^d)$, and assume that
$\bsy^* \in \bbR^d$ is such that
\begin{align}\label{eq:assumptions2}
 \frac{\partial \phi}{\partial y_1}(\bsy^*) = 0, \quad
 \nabla \phi(\bsy^*)\ne \bszero,\quad\mbox{and}\quad
  \nabla \frac{\partial \phi}{\partial y_1}(\bsy^*)\ne \bszero,
\end{align}
with $\nabla \phi(\bsy^*)$ not parallel to $ \nabla (\partial
\phi/\partial y_1)(\bsy^*)$. Then for any $t$ in some open interval containing
$\phi(\bsy^*)$ there exists a point $\bsy^{(t)} \in \bbR^d$
in a neighborhood of $\bsy^*$ at which
\begin{align}
\label{eq:y^t}
 \phi(\bsy^{(t)}) = t, \quad
 \frac{\partial \phi}{\partial y_1}(\bsy^{(t)}) = 0,
 \quad\mbox{and}\quad
 \nabla \phi(\bsy^{(t)})\ne \bszero.
\end{align}
Moreover, if we assume also that $(\partial^2 \phi/\partial y_1^2)(\bsy^*)
\ne 0$, then the preintegrated quantity  $(P_1 f_t) (\bsy_{-1})$ has a
square-root singularity at $\bsy^{(t)}_{-1} \in \bbR^{d-1}$ along any line
in $\bbR^{d-1}$ through $\bsy^{(t)}_{-1}$ that is not orthogonal to
$\nabla\phi(\bsy^{(t)})$.
\end{theorem}

\begin{proof}
It is convenient to define $\psi(\bsy):= (\partial \phi/\partial
y_1)(\bsy)$, which by assumption is a real-valued $C^1(\mathbb{R}^d)$
function that satisfies
\[
\psi(\bsy^*) = 0 \quad \mbox{and}\quad \nabla \psi(\bsy^*) \ne \bszero.
\]
We need to show that for $t$ in some open interval containing
$\phi(\bsy^*)$ there exists $\bsy^{(t)}\in \mathbb{R}^d$ in a neighborhood
of $\bsy^*$ at which
\begin{align*}
 \phi(\bsy^{(t)}) = t, \quad
 \psi(\bsy^{(t)}) = 0,
 \quad\mbox{and}\quad
 \nabla \phi(\bsy^{(t)})\ne \bszero.
\end{align*}
Clearly, we can confine our search for $\bsy^{(t)}$ to the zero level set
of $\psi$, that is, to the solutions of
\begin{equation*} 
\psi(\bsy) = 0, \quad \bsy \in \mathbb{R}^d.
\end{equation*}
Since $\nabla \psi(\bsy)$ is continuous and  non-zero in a neighborhood of
$\bsy^*$, the zero level set of $\psi$ is a manifold of dimension $d-1$
near $\bsy^*$, whose tangent hyperplane at $\bsy^*$ is orthogonal to
$\nabla \psi (\bsy^*)$, see, e.g., \cite[Chapter 5]{Lee00}.  On this
hyperplane there is a search direction starting from $\bsy^*$ for which
$\phi(\bsy)$ has a maximal rate of increase, namely the direction of the
orthogonal projection of $\nabla \phi(\bsy^*)$ onto the tangent
hyperplane, noting that this is a non-zero vector because
$\nabla\phi(\bsy^*)$ is not parallel to $\nabla \psi (\bsy^*)$. Setting
out from the point~$\bsy^*$ in the direction of positive gradient, the
value of $\phi$ is strictly increasing in a sufficiently small
neighborhood of~$\bsy^*$, while in the direction of negative gradient it
is strictly decreasing. Thus searching on the manifold for a $\bsy^{(t)}$
such that $\phi(\bsy^{(t)}) = t$ will be successful in one of these
directions for $t$ in a sufficiently small open interval containing
$\phi(\bsy^*)$.

Under the additional assumption that $(\partial^2 \phi / \partial
y_1^2)(\bsy^*) \ne 0$, all the conditions of Theorem~\ref{thm:sqrt} are
satisfied with $\bsy^*$ replaced by $\bsy^{(t)}$, noting that because
$\phi \in C^2(\mathbb{R}^d)$ the second derivative is also non-zero in a
sufficiently small neighborhood of $\bsy^*$. This completes the proof.
\end{proof}

\begin{remark}
\label{rem:not-iso} We now show that for $\phi$ as in
Examples~\ref{eg1}--\ref{eg4} the singularities in $P_1 f_t$, with $f_t$
as in \eqref{general}, are not isolated and furthermore, we give the exact
regions in which the singularities exist.

\begin{itemize}
\item For $\phi$ as in Example~\ref{eg1} we may choose $\bsy^* = (0,
    0)$, as in Remark~\ref{rem:sqrt}. Indeed, the gradient of the
    first derivative with respect to $y_1$ is $\nabla ( \partial \phi
    /\partial y_1) (y_1, y_2) = (-2, 0)$, which is not parallel to
    $\nabla \phi (y_1, y_2) = (-2y_1, 1)$ for all $(y_1, y_2) \in
    \R^2$. It follows that \eqref{eq:assumptions2} holds, e.g.,
    at~$\bsy^*=(0,0)$. Hence Theorem~\ref{thm:not-iso} implies that
    for $t$ in some interval around $\phi(\bsy^*) = 0$ there is
    $\bsy^{(t)}= (y_1^{(t)}, y_2^{(t)})$ in a neighborhood of $\bsy^*
    = (0, 0)$ such that $\phi(\bsy^{(t)}) = t$ and \eqref{eq:y^t}
    holds. In particular, there is still a square-root singularity in
    $(P_1 f_t)(y_2)$ at $y_2 = y_2^{(t)}$. We confirm that this is
    indeed the case by taking $\bsy^{(t)} = (0, t)$ and by observing
    that, as is easily verified,  $(P_1f_t) (y_2)$ has a square-root
    singularity at $y_2 = t$ for all real numbers $t$. This
    singularity in $P_1 f_t$ is similar to the singularity depicted in
    Figure~\ref{fig:eg1}, but translated by $t$.

\item For $\phi$ as in Example~\ref{eg3} we can consider $\bsy^* = (0,
    \pm 1)$ as in Remark~\ref{rem:sqrt}. The gradient of the first
    derivative with respect to $y_1$ is $\nabla ( \partial \phi
    /\partial y_1) (y_1, y_2) = (-2, 0)$, which is not parallel to
    $\nabla \phi(y_1, y_2) = (-2y_1, 2y_2)$ for all $(y_1, y_2) \in
    \R^2$ with $y_2 \neq 0$. Thus \eqref{eq:assumptions2} holds, e.g.,
    at $\bsy^* = (0, \pm1)$. Theorem~\ref{thm:not-iso} implies that
    for $t$ in some interval around $ \phi(\bsy^*) = 1$ there is a
    point $\bsy^{(t)}$ in a neighborhood of $\bsy^* = (0,\pm1)$ such
    that $\phi(\bsy^{(t)}) = t$, \eqref{eq:y^t} holds, and
    $(P_1f_t)(y_2)$ has a square-root singularity at $y_2 =
    y_2^{(t)}$. Indeed, for any real number $t> 0$, taking $\bsy^{(t)}
    = (0, \pm \sqrt{t})$ gives $\phi(\bsy^{(t)}) = t$, and it can
    easily be verified that $(P_1 f_t)(y_2)$ has two square-root
    singularities at $y_2 = \pm \sqrt{t}$. In this case the behavior
    of $P_1 f_t$ is similar to Figure~\ref{fig:eg2}, with the location
    of the singularities now depending on $t$.

\item Since $\phi$ from Example~\ref{eg2} is simply a translation of
    Example~\ref{eg3} by $-1$, similar singularities exist for that
    case for $t > -1$.

\item For $\phi$ as in Example~\ref{eg4} the condition
    \eqref{eq:assumptions2} never holds, since $\nabla(\partial
    \phi/\partial y_1)(y_1,y_2)= (0,0)$ whenever $(\partial
    \phi/\partial y_1)(y_1,y_2)= 0$. So no conclusion can be drawn
    from Theorem~\ref{thm:not-iso} in this case.  It is no
    contradiction that, as is easily seen, there is a singularity (of
    cube-root character) in $(P_1 f_t)(y_2)$ at $y_2 = t$ for every
    real number~$t$.

\end{itemize}
\end{remark}

\section{A high-dimensional example} \label{sec:num}

Motivated by applications in computational finance, for a high-dimensional
example we consider the problem of approximating the fair price for a
\emph{digital Asian option}, a problem that can be formulated as an
integral as in \eqref{problem1} with a discontinuous integrand of the form
\eqref{general}. When monotonicity holds, it was shown in \cite{GKLS18}
that preintegration not only has theoretical smoothing benefits, but also
that when followed by a QMC rule to compute the $(d-1)$-dimensional
integral the computational experience can be excellent. On the other hand
that paper provided no insight as to what happens, either theoretically or
numerically, when monotonicity fails. In this section, in contrast, we
will deliberately apply preintegration using a chosen variable for which
the monotonicity condition fails.  We will demonstrate the resulting lack
of smoothness, using the theoretical results from the previous section,
and show that the performance of the subsequent QMC rule can degrade when
the preintegration variable lacks the monotonicity property.

For a given strike price $K$, the payoff for a digital Asian call option
is given by
\[
\mathrm{payoff} \,=\, \GKSind(\phi - K),
\]
where $\phi$ is the average price of the underlying stock over the time period.
Under the Black--Scholes model the time-discretised average is given by
\begin{equation}
\label{eq:phi-opt}
 \phi(\bsy) \,=\, \frac{S_0}{d} \sum_{k = 1}^d
 \exp \bigg((r - \tfrac{1}{2}\sigma^2)\frac{kT}{d} + \sigma \bsA_k \bsy\bigg),
\end{equation}
where $\bsy = (y_k)_{k = 1}^d$ is a vector of i.i.d.\ standard normal
random variables, $S_0$ is the initial stock price, $T$ is the final time,
$r$ is the risk-free interest rate, $\sigma$ is the volatility and $d$ is
the number of timesteps, which is also the dimension of the problem. Note
that in \eqref{eq:phi-opt} we have already made a change of variables to
write the problem in terms of standard normal random variables, by
factorising the covariance matrix of the Brownian motion increments
as $\Sigma = AA^\top$, where the entries of the covariance matrix are
$\Sigma_{k, \ell} = \min(k, \ell) \times T/d$. Then in \eqref{eq:phi-opt},
$\bsA_k$ denotes the $k$th row of this matrix factor.

The fair price of the option is then given by the discounted expected
payoff
\begin{equation}
\label{eq:price}
 e^{-rT}\,\bbE[\mathrm{payoff}] \,=\,
 e^{-rT} \int_{\bbR^d} \GKSind(\phi(\bsy) - K) \rho_d(\bsy) \, \mathrm{d} \bsy\,.
\end{equation}
Letting $f(\bsy) = \GKSind(\phi(\bsy) - K)$, this example clearly fits
into the framework \eqref{general} where $\phi$ is the average stock price
\eqref{eq:phi-opt}, $t$ is the strike price $K$ and each $\rho$ is a
standard normal density.

There are three popular methods for factorising the covariance matrix: the
\emph{standard construction} (which uses the Cholesky factorisation),
\emph{Brownian bridge}, and \emph{principal components} or PCA, see, e.g.,
\cite{Glasserman} for further details. In the first two cases all
components of the matrix $A$ are positive, in which case it is easily seen
by studying the derivative of \eqref{eq:phi-opt} with respect to $y_j$ for
some $j=1,\ldots,d$,
\begin{equation*}
\frac{\partial\phi}{\partial y_j} (\bsy) \,=\,
\frac{S_0}{d} \sum_{k = 1}^d \sigma A_{k, j}
\exp \bigg((r - \tfrac{1}{2}\sigma^2)\frac{kT}{d} + \sigma \bsA_k \bsy\bigg),
\end{equation*}
that $\phi$ is monotone increasing with respect to $y_j$ no matter which
$j$ is chosen.

In contrast, for the PCA construction, which we consider below, the
situation is very different, in that there is only one choice of $j$ for
which $\phi$ is monotone with respect to $y_j$. This is because with PCA
the factorisation of $\Sigma$ employs its eigendecomposition, with the
$j$th column of $A$ being a (scaled) eigenvector corresponding to the
$j$th eigenvalue labeled in decreasing order.
Since the covariance matrix $\Sigma$ is
positive definite the eigenvector corresponding to the largest
eigenvalue has all components positive, thus for $j = 1$ monotonicity of
$\phi$ is achieved. On the other hand, every eigenvector other than the
first is orthogonal to the first, and therefore must have components of
both signs. Given that
\begin{align*}
 & A_{k,j} > 0 \implies \exp(A_{k,j}y_j) \to
 \begin{cases}
 +\infty & \mbox{as } y_j \to +\infty,\\
 0 & \mbox{as } y_j \to -\infty,
 \end{cases}\\
 & A_{k,j} < 0 \implies \exp(A_{k,j}y_j) \to
 \begin{cases}
 0 & \mbox{as } y_j \to +\infty,\\
 +\infty & \mbox{as } y_j \to -\infty,
 \end{cases}
\end{align*}
it follows that for $j \ne 1$ there is at least one term in the sum over
$k$ in \eqref{eq:phi-opt} that approaches $+\infty$ as $y_j\to +\infty$
and at least one other term that approaches $+\infty$ as $y_j \to
-\infty$. Given that all terms in the sum over $k$ in \eqref{eq:phi-opt}
are positive, it follows that for the PCA case and $j \ne 1 $, $\phi$ must
approach $+\infty$ as $y_j \to \pm \infty$, so is definitely not monotone.
Moreover, with respect to each variable $y_j$ the function $\phi$ is
strictly convex, since
\begin{equation*} 
\frac{\partial^2\phi}{\partial y_j^2} (\bsy) \,=\,
\frac{S_0}{d} \sum_{k = 1}^d (\sigma A_{k, j})^2
\exp \bigg((r - \tfrac{1}{2}\sigma^2)\frac{kT}{d} + \sigma \bsA_k \bsy\bigg)
\,>\, 0 \quad \text{for all } \bsy \in \R^d.
\end{equation*}

For definiteness, in the following discussion we denote by $y_2$ the
preintegration variable for which monotonicity fails, and denote the other
variables by $\bsy_{-2} = (y_1, y_3, \ldots, y_d)$.
We now use the results from the previous section to determine the
smoothness, or rather the lack thereof, of $P_2f$ when $f(\bsy) =
\GKSind(\phi(\bsy) - K)$. To do this we will use Theorem~\ref{thm:sqrt}
and Theorem~\ref{thm:not-iso}, where with a slight abuse of notation we
replace $y_1$ by $y_2$ as our special preintegration variable.

First, note that we have already established that $\phi$ is not monotone
with respect to $y_2$, and since $\phi$ is strictly convex with respect to
$y_2$, for a given $\bsy_{-2}$ there exists a unique $y_2^* \in \bbR$ such
that $(\partial \phi/\partial y_2)(y_2^*,\bsy_{-2}) = 0$. Since $\phi$ is
strictly increasing with respect to $y_1$, it follows that $\nabla
\phi(y_2^*,\bsy_{-2}) \neq \bszero$. Furthermore, since
$(\partial^2\phi/\partial y_2^2)(y_2^*,\bsy_{-2}) > 0$,
Theorem~\ref{thm:sqrt} implies that for $K = \phi(y_2^*,\bsy_{-2})$ the
preintegrated function $P_2f$ has a square-root singularity along any line
not orthogonal to $\nabla_{\!-2} \phi(y_2^*,\bsy_{-2})$, with
$\nabla_{\!-2}$ defined analogously to $\nabla_{\!-1}$ in
Theorem~\ref{thm:sqrt}.

Furthermore, Theorem~\ref{thm:not-iso} implies that this singularity is
not isolated. To apply Theorem~\ref{thm:not-iso} we note that we have
already established the first two conditions in \eqref{eq:assumptions2}
(recall again that we now  take $y_2$ as the preintegration variable). We
also have $(\partial^2 \phi / \partial y_2^2)(y_2^*,\bsy_{-2}) > 0$, which
implies $\nabla (\partial \phi / \partial y_2)(y_2^*,\bsy_{-2}) \neq
\bszero$. Moreover, we know that $\nabla (\partial \phi /
\partial y_2)(y_2^*,\bsy_{-2}) $ and $\nabla \phi(y_2^*,\bsy_{-2})$ are not
parallel, since the former has a positive second component while the
latter has a zero second component.

\begin{figure}[t]
\begin{center}
\includegraphics[width=\textwidth]{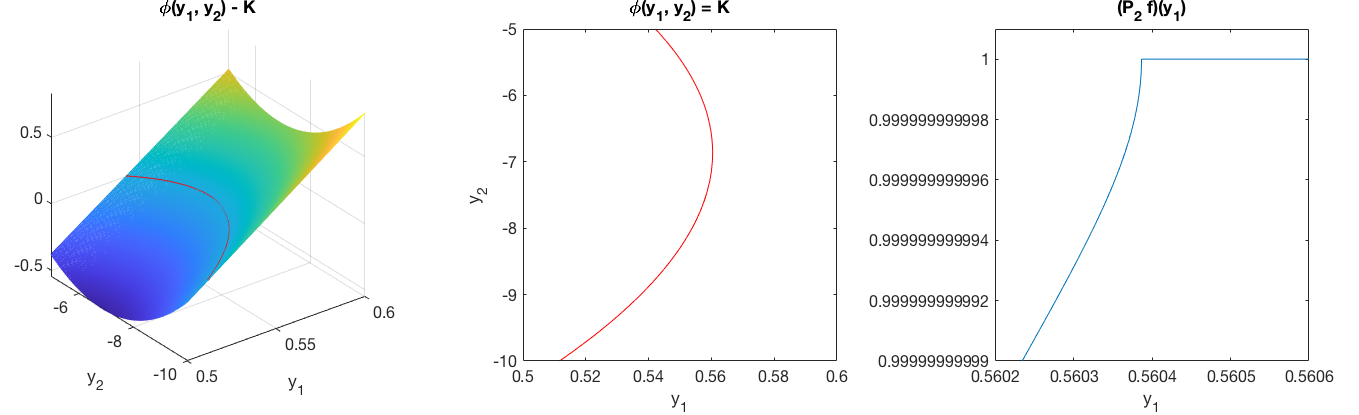}
\end{center}
\caption{Illustrations for digital Asian option in two dimensions.} \label{fig:opt-sing}
\end{figure}

To visualise this singularity, in Figure~\ref{fig:opt-sing} we provide an
illustration of the option in two dimensions. (Note that we consider $d =
2$ here for visualisation purposes only; we have shown already that the
singularity exists for any choice of $d > 1$. Later we present numerical
results for $d = 256$.) Figure~\ref{fig:opt-sing} gives a contour plot of
$\phi(y_1, y_2) - K$ (left), the zero level set of $\phi(y_1, y_2) = K$
(middle) and then the graph of $P_2 f$ (right). As expected, we can
clearly see that $P_2 f$ has a singularity that is of square-root nature.

To perform the preintegration step $P_2f$ in practice,
note that since $\phi$ is strictly convex with respect to $y_2$
for each $\bsy_{-2} \in \bbR^{d - 1}$ there is a single
turning point $y_2^* \in \bbR$ for which $(\partial \phi /
\partial y_2) (y_2^*, \bsy_{-2}) = 0$ and $\phi(y_2^*, \bsy_{-2})$ is a global minimum.
It follows that there are at most \emph{two} distinct points,
$\xi_a(\bsy_{-2}) \leq \xi_b(\bsy_{-2})$, such that
\[
 \phi(\xi_a(\bsy_{-2}), \bsy_{-2}) \,=\, K \,=\, \phi(\xi_b(\bsy_{-2}), \bsy_{-2})\,.
\]
Preintegration with respect to $y_2$ then simplifies to
\begin{align*}
(P_2 f)(\bsy_{-2})
\,&=\, \int_{-\infty}^\infty \GKSind(\phi(y_2, \bsy_{-2}) - K)\, \rho(y_2) \, \mathrm{d} y_2
\\
&=\, \begin{cases}
1 & \text{if } \phi(y_2^*, \bsy_{-2}) \geq K, \\[2mm]
\displaystyle \int_{-\infty}^{\xi_a(\bsy_{-2})} \rho(y_2) \, \mathrm{d} y_2
+ \int_{\xi_b(\bsy_{-2})}^\infty \rho(y_2) \, \mathrm{d} y_2
& \text{otherwise},
\end{cases}
\\
&=\, \begin{cases}
1 & \text{if } \phi(y_2^*, \bsy_{-2}) \geq K, \\[2mm]
\Phi(\xi_a(\bsy_{-2})) + 1 - \Phi(\xi_b(\bsy_{-2})) \hspace{1.5cm}
& \text{otherwise}.
\end{cases}
\end{align*}
In practice, for each $\bsy_{-2} \in \bbR^{d - 1}$ the turning point
$y_2^*$ and the points of discontinuity $\xi_a(\bsy_{-2})$ and
$\xi_b(\bsy_{-2})$ are computed numerically, e.g., by Newton's method.

We now look at how this lack of smoothness affects the performance of
using a numerical preintegration method to approximate the fair price for
the digital Asian option in $d = 256$ dimensions. Explicitly, we
approximate the integral in \eqref{eq:price} by applying a $(d -
1)$-dimensional QMC rule to $P_2 f$. As a comparison, we also present
results for approximating the integral in \eqref{eq:price} by applying the
same QMC rule to $P_1 f$. Recall that $\phi$ is monotone in dimension $1$
and furthermore, it was shown in \cite{GKLS18} that $P_1 f$ is smooth.

For the QMC rule we use a randomly shifted lattice rule based on the
generating vector \texttt{lattice-32001-1024-1048576.3600} from
\cite{Kuo-lattice} using $N = 2^{10}, 2^{11}, \ldots, 2^{19}$ points with
$R = 16$ random shifts. The parameters for the option are $S_0 = \$100$,
$K = \$110$, $T = 1$, $d = 256$ timesteps, $r = 0.1$ and $\sigma = 0.1$.
We also performed a standard Monte Carlo approximation using $R \times N$
points and a plain (without preintegration) QMC approximation using the
same generating vector.

In Figure~\ref{fig:converg-N}, we plot the convergence of the standard
error, which we estimate by the sample standard error over the different
random shifts, in terms of the total number of function evaluations $R
\times N$. We can clearly see that preintegration with respect to $y_2$
produces less accurate results compared to preintegration with respect
to~$y_1$, with errors that are up to an order of magnitude larger for the
higher values of~$N$. We also note that to achieve a given error, say
$10^{-4}$, the number of points needs to be increased tenfold.
Furthermore, we observe that the empirical convergence rate for
preintegration with respect to $y_2$ is $N^{-0.9}$, which is sightly worse
than the rate of $N^{-0.98}$ for preintegration with respect to $y_1$.
Hence, when the monotonicity condition fails, not only does the theory for
QMC fail due to the presence of a singularity, but we also observe worse
results in practice and somewhat slower convergence.

We also plot the error of standard Monte Carlo and QMC approximations, which
behave as expected and are both significantly outperformed by the two preintegration methods.

\begin{figure}[t]
\centering
\includegraphics[scale=0.5,trim={18 4 40 5},clip=true]{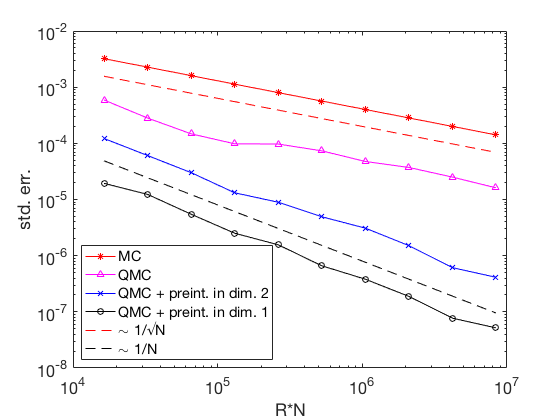}
\caption{Convergence in $N$ for different approximations of the fair price for a digital Asian option.}
\label{fig:converg-N}
\end{figure}

\section{Conclusion} \label{sec:conc}

If the monotonicity property \eqref{monotone1} fails and $f = f_t$ is
defined by \eqref{general} then we have seen that generically there is a
singularity in $P_1 f$ for some values of $t$, and under known conditions
this is true even for values of $t$ in an interval.

It should also be noted that implementation of preintegration is more
difficult if monotonicity fails, since instead of a single integral from
$\xi(\bsy_{-1})$ to $\infty$, as in \eqref{P1psi}, there will in
general be additional finite or infinite intervals to integrate over, all
of whose end points must be discovered by the user for each required value
of $\bsy_{-1}$.

To explore the consequences of a lack of monotonicity empirically, we
carried out in Section~\ref{sec:num} a $256$-dimensional calculation of
pricing a digital Asian option, first by preintegrating with respect to a
variable known to lack the monotonicity property, and then with a variable
where the property holds, with the result that both accuracy and rate of
convergence were observed to be degraded when monotonicity fails.

There is an additional problem of preintegrating with respect to a
variable for which the monotonicity fails, namely that because of the
proven lack of smoothness, the resulting preintegrated function no longer
belongs to the space of $(d-1)$-variate functions of dominating mixed
smoothness of order one, and as a consequence there is at present no
theoretical support for our use of QMC integration for this
$(d-1)$-dimensional integral.

The practical significance of this paper is that effective use of
preintegration is greatly enhanced by the preliminary identification of a
special variable for which the monotonicity property is known to hold. The
paper does not offer guidance on the choice of variable if there is more
than one such variable; in such cases it may be natural to choose the
variable for which preintegration leads to the greatest reduction in
variance.

\section*{Acknowledgements}

The authors acknowledge the support of the Australian Research Council
under the Discovery Project DP210100831.


\begin{thebibliography}{99}
\bibitem{ACN13a}
 N.~Achtsis, R.~Cools, and D.~Nuyens,
 \emph{Conditional sampling for barrier option pricing under the LT method},
 SIAM J.\ Financial Math. \textbf{4} (2013), 327--352.

\bibitem{ACN13b}
 N.~Achtsis, R.~Cools, and D.~Nuyens,
 \emph{Conditional sampling for barrier option pricing under the Heston
 model}, in: J.~Dick, F.Y. Kuo, G.W. Peters, I.H. Sloan (eds.), {M}onte {C}arlo
 and Quasi-{M}onte {C}arlo Methods 2012, pp. 253--269, Springer-Verlag,
 Berlin/Heidelberg (2013).

\bibitem{BST17}
 C.~Bayer, M.~Siebenmorgen, and R.~Tempone,
 \emph{Smoothing the payoff for efficient computation of Basket option
 prices}, Quant.\ Finance \textbf{18} (2018), 491--505.

\bibitem{BG04}
 H.~Bungartz and M.~Griebel,
 \emph{Sparse grids}, Acta Numer.\ \textbf{13} (2004), 147--269.

\bibitem{DKS13}
  J.~Dick, F.~Y. Kuo, and I.~H. Sloan,
  \emph{High dimensional integration -- the quasi-Monte Carlo way},
  Acta Numer.\ 22 (2013), 133--288.

\bibitem{GKS21}
 A.~D.~Gilbert, F.~Y.~Kuo, and I.~H.~Sloan,
 \emph{Approximating distribution functions and densities using
quasi-Monte Carlo methods after smoothing by preintegration},
arXiv:2112.10308, (2021).

\bibitem{Glasserman} P.~Glasserman, Monte Carlo Methods in Financial
    Engineering, Springer-Verlag, Berlin-Heidelberg, (2003).

\bibitem{GlaSta01}
 P.~Glasserman and J.~Staum, \emph{Conditioning on one-step survival for
 barrier option simulations}, Oper.\ Res., \textbf{49} (2001), 923--937.

\bibitem{GKS10} M.~Griebel, F.~Y.~Kuo, and I.~H.~Sloan, \emph{The
    smoothing effect of the ANOVA decomposition},
    J.\ Complexity\ \textbf{26} (2010), 523--551.

\bibitem{GKS13} M.~Griebel, F.~Y.~Kuo, and I.~H.~Sloan, \emph{The
    smoothing effect of integration in $\bbR^d$ and the ANOVA decomposition},
    Math.\ Comp.\ \textbf{82} (2013), 383--400.

\bibitem{GKSnote} M.~Griebel, F.~Y. ~Kuo, and I.~H.~Sloan, \emph{Note on
    ``the smoothing effect of integration in~$\bbR^d$ and the ANOVA
    decomposition''}, Math.\ Comp.\ \textbf{86} (2017), 1847--1854.

\bibitem{GKLS18}  A.~Griewank, F.~Y.~Kuo, H.~Le\"ovey, and I.~H.~Sloan.
    \emph{High dimensional integration of kinks and
    jumps -- smoothing by preintegration},
    J. Comput. Appl. Math. \textbf{344} (2018),  259--274.

\bibitem{Hol11}
  M.~Holtz,
  Sparse Grid Quadrature in High Dimensions with Applications in Finance and Insurance
  (PhD thesis), Springer-Verlag, Berlin, 2011.

\bibitem{Kuo-lattice} F.~Y.~Kuo,
    \texttt{https://web.maths.unsw.edu.au/$\sim$fkuo/lattice/index.html},
    (2007), accessed December 13, 2021.


\bibitem{LPB21}  P.~L'Ecuyer,  F.~Puchhammer, and A.~Ben Abdellah.
    \emph{Monte Carlo and Quasi-Monte Carlo Density Estimation via Conditioning},
    arXiv:1906.04607, (2021). 

\bibitem{Lee00} J.~.M.~Lee,
Introduction to Smooth Manifolds,
Springer, New York, (2000).

\bibitem{NuyWat12}
 D.~Nuyens and B.~J.~Waterhouse,
 \emph{A global adaptive quasi-Monte Carlo algorithm for functions of low truncation dimension
 applied to problems from finance}, in: L.~Plaskota and H.~Wo\'zniakowski (eds.), {M}onte {C}arlo
 and Quasi-{M}onte {C}arlo Methods 2010, pp. 589--607, Springer-Verlag,
 Berlin/Heidelberg (2012).

\bibitem{WWH17}
 C.~Weng, X.~Wang, and Z.~He,
 \emph{Efficient computation of option prices and greeks by quasi-Monte Carlo
 method with smoothing and dimension reduction},
 SIAM J.\ Sci.\ Comput.\ \textbf{39} (2017), B298--B322.

\end{thebibliography}
\end{document}